\documentclass{amsart}
\usepackage{graphicx} 
\usepackage{arydshln}
\usepackage{hyperref}
\usepackage{txfonts, amsmath,amstext,amsthm,amscd,amsopn,verbatim,amssymb, amsfonts}
\usepackage{fullpage}
\usepackage{todonotes}

\usepackage[bbgreekl]{mathbbol}
\usepackage{enumerate}

\usepackage{float}
\restylefloat{table}

\usepackage{tikz}
\usepackage{tikz-cd}
\usetikzlibrary{matrix}
\usetikzlibrary{shapes}
\usetikzlibrary{arrows}
\usetikzlibrary{calc,3d}
\usetikzlibrary{decorations,decorations.pathmorphing,decorations.pathreplacing,decorations.markings}
\usetikzlibrary{through}

\tikzset{snakeit/.style={decorate, decoration={snake, amplitude=.2mm,segment length=1mm}}}

\tikzset{ext/.style={circle, draw,inner sep=1pt}, int/.style={circle,draw,fill,inner sep=2pt},nil/.style={inner sep=1pt}}
\tikzset{cy/.style={circle,draw,fill,inner sep=2pt},scy/.style={circle,draw,inner sep=2pt},scyx/.style={draw,cross out,inner sep=2pt},scyt/.style={draw,regular polygon,regular polygon sides=3,inner sep=0.95pt}}
\tikzset{exte/.style={circle, draw,inner sep=3pt},inte/.style={circle,draw,fill,inner sep=3pt}}
\tikzset{diagram/.style={matrix of math nodes, row sep=3em, column sep=2.5em, text height=1.5ex, text depth=0.25ex}}
\tikzset{diagram2/.style={matrix of math nodes, row sep=0.5em, column sep=0.5em, text height=1.5ex, text depth=0.25ex}}
\tikzset{rowcolsep/.style={column sep=.2cm, row sep=.1cm}}

\tikzset{
  crossed/.style={
    decoration={markings,mark=at position .5 with {\arrow{|}}},
    postaction={decorate},
    shorten >=0.4pt}}

\tikzset{every picture/.style={baseline=-.65ex} }

\theoremstyle{plain}
  \newtheorem{thm}{Theorem}
  \newtheorem{defi}[thm]{Definition}
  \newtheorem{prop}[thm]{Proposition}
  
  \newtheorem{cor}[thm]{Corollary}
  \newtheorem{conjecture}[thm]{Conjecture}
  \newtheorem{lemma}[thm]{Lemma}
\theoremstyle{definition}
  \newtheorem{ex}{Example}
  \newtheorem{rem}{Remark}

\newcommand{\alg}[1]{\mathfrak{{#1}}}

\newcommand{\ad}{{\text{ad}}}

\newcommand{\FreeLie}{\mathrm{FreeLie}}

\newcommand{\bpm}{\begin{pmatrix}}
\newcommand{\epm}{\end{pmatrix}}

\newcommand{\cy}{\mathfrak{c}}

\DeclareMathOperator{\gr}{gr}
\newcommand{\tder}{\alg{tder}}
\newcommand{\sder}{\alg{sder}}
\newcommand{\kv}{\alg{kv}}

\newcommand{\grt}{\alg {grt}}

\DeclareMathOperator{\vdim}{dim}

\newcommand{\Q}{\mathbb{Q}}

\usepackage{tikz-cd}
\tikzset{%
    symbol/.style={%
        draw=none,
        every to/.append style={%
            edge node={node [sloped, allow upside down, auto=false]{$#1$}}}
    }
}

\DeclareMathOperator{\Log}{Log}

\newcommand{\BK}{\mathrm{BK}}

\newcommand{\f}{\mathfrak{f}}
\newcommand{\as}{\mathfrak{a}}

\newcommand{\lkv}{\mathfrak{lkv}}
\newcommand{\free}{\mathfrak{free}}
\newcommand{\ds}{\mathfrak{ds}}
\newcommand{\lds}{\mathfrak{lds}}
\newcommand{\hlkv}{\widehat{\mathfrak{lkv}}}
\newcommand{\hkv}{\widehat{\mathfrak{kv}}}
\title{Numerical computation of linearized KV and the Deligne-Drinfeld and Broadhurst-Kreimer conjectures}

\author{Florian Naef}
\address{School of Mathematics, Trinity College, Dublin 2, Ireland }
\author{Thomas Willwacher }
\address{Department of Mathematics, ETH Zurich, Rämistrasse 101, 8092 Zurich, Switzerland}

\thanks{This work has been partially supported by the NCCR Swissmap, funded by the Swiss National Science Foundation}

\begin{document}
\begin{abstract}
We compute numerically the dimensions of the graded quotients of the linearized Kashiwara-Vergne Lie algebra $\hlkv_2$ in low weight, confirming a conjecture of Raphael-Schneps in those weights.
The Lie algebra $\hlkv_2$ appears in a chain of inclusions of Lie algebras, including also the linearized double shuffle Lie algebra and the (depth associated graded of the) Grothendieck-Teichmüller Lie algebra. Hence our computations also allow us to check the validity of the Deligne-Drinfeld conjecture on the structure of the Grothendieck-Teichmüller group up to weight 29, and (a version of) the the Broadhurst-Kreimer conjecture on the number of multiple zeta values for a range of weight-depth pairs significantly exceeding the previous bounds.
Our computations also verify a conjecture by Alekseev-Torossian on the Kashiwara-Vergne Lie algebra up to weight 29.
\end{abstract}

\maketitle

\section{Introduction}

We consider the following well-studied chain of inclusions of Lie algebras (\cite{Brown, AlekseevTorossian, Furusho, FurushoFour, Schneps})
\begin{equation}\label{equ:chain}
\free \subset \grt_1 \subset \ds \subset \hkv_2,
\end{equation}
where $\free$ is a free Lie algebra in generators $\sigma_{2k+1}$, $k=1,2,\dots$, $\grt_1$ is the Grothendieck-Teichmüller Lie algebra introduced by Drinfeld \cite{Drinfeld}, $\ds$ is Racinet's double shuffle Lie algebra, and $\hkv_2$ is the Kashiwara-Vergne Lie algebra of Alekseev-Torossian \cite{AlekseevTorossian}.
These Lie algebras all govern objects of high interest in (arguably) distinct areas of mathematics, which are hence linked by the above chain. Concretely, $\free$ is (essentially) the Lie algebra of the Galois group of the category $\mathit{MTM}(\mathbb Z)$ of mixed Tate motives over $\mathbb Z$. The Grothendieck-Teichmüller group is the homotopy automorphism group of the rationalized little 2-disks operad and appears in various places in algebraic geometry and algebraic topology. The double shuffle Lie algebra $\ds$ governs the space of formal solutions of the double-shuffle relations, that are in particular satisfied by the multiple zeta values (MZVs). Knowledge of $\ds$ hence yields bounds on the number of $\Q$-linearly independent MZVs, an object of high interest in number theory and algebraic geometry. Finally, $\hkv_2$ controls the space of solutions of the Kashiwara-Vergne conjecture in Lie theory.

The four Lie algebras in \eqref{equ:chain} carry a grading by weight and additionally a filtration by depth, and the inclusions are compatible with both the weight grading and the depth filtration.
In addition, there is a "linearized" variant $\lds$ of the double shuffle Lie algebra $\ds$, obtained by keeping only the terms of minimal depth in the defining equations \cite{BrownDepth}. Similarly, there is also a linearized version $\hlkv_2$ of the Kashiwara-Vergne Lie algebra $\hkv_2$ introduced by Raphael-Schneps \cite{RaphaelSchneps}, which again may be defined by dropping terms of higher depth in the defining equations, see section \ref{sec:kv def} below.
\begin{thm}[Raphael-Schneps]\label{thm:rs}
We have a diagram of inclusions of bigraded (by weight and depth) Lie algebras
\begin{equation}
\label{equ:rsdiag}
\begin{tikzcd}
\gr \ds \ar[hookrightarrow]{r}\ar[hookrightarrow]{d} & \gr \hkv_2 \ar[hookrightarrow]{d}\\
\lds \ar[hookrightarrow]{r} & \hlkv_2
\end{tikzcd}.
\end{equation}
Here $\gr(-)$ denotes the associated gradded with respect to the depth filtration.
Furthermore, the inclusions are isomorphisms in depth $\leq 3$.
\end{thm}

There are several conjectures surrounding our 4 Lie algebras that have been raised by various authors over the years. Combined, they state the following: 
\begin{conjecture}[Deligne-Drinfeld-Ihara, Alekseev-Torossian, Schneps-Raphael]
\label{conj:all iso}
    The inclusions of weight-graded Lie algebras in \eqref{equ:chain} and in \eqref{equ:rsdiag} are all isomorphisms in weights $W\geq 2$.
\end{conjecture}

Furthermore, there is a conjecture of Broadhurst-Kreimer on the number of linearly independent multiple zeta values in each depth. To state it for $\lds$, let us define the numbers $\BK_{W,D}$ as the Taylor coefficients of the generating function
\begin{equation}\label{equ:BK fun}
\sum_{W,D} s^W t^D \BK_{W,D}
:=
\Log\left(
\frac{1}{1-\frac{s^3t}{1-s^2} - \frac{s^{12}(t^2-t^4)}{(1-s^4)(1-s^6)}}
\right) 
\end{equation}
where $s,t$ are formal variables and $\Log$ is the plethystic logarithm 
\[
\Log(f)(s,t) = 
\sum_{\ell\geq 1}
\frac{\mu(\ell)}{\ell}
\log(f(s^\ell, t^\ell))
\]
with $\mu(-)$ the Moebius function. Then we have:

\begin{conjecture}[Broadhurst-Kreimer conjecture for $\lds$]
\label{conj:BK}
The dimension of the bigraded component $\lds^{(W,D)}\subset \lds$ of weight $W$ and depth $D$ has dimension $\BK_{W,D}$.
\end{conjecture}

In view of Conjecture \ref{conj:all iso} it is natural to also state a version of this conjecture for $\hlkv_2$.

\begin{conjecture}[Broadhurst-Kreimer conjecture for $\hlkv_2$]
\label{conj:BK 2}
The dimension of the bigraded component $\hlkv_2^{(W,D)}\subset \hlkv_2$ of weight $W$ and depth $D$ is $\BK_{W,D}$.
\end{conjecture}

The contribution of this work is the numeric computation of the numbers $\vdim( \hlkv_2^{(W,D)})$ for low weight $W$ and depth $D$.
The result is shown in Figure \ref{fig:results} below. 

\begin{thm}[Numeric computation]\label{thm:main}
    The dimensions $\mathrm{dim}( \hlkv_2^{(W,D)})$ of the bigraded components of $\hlkv_2$ of weight $W$ and degree $D$ are as shown in Figure \ref{fig:results}. 
    For all displayed entries $(W,D)$ the space $\hlkv_2^{(W,D)}$ is spanned by the images of elements of $\lds^{(W,D)}$ under the inclusion $\lds^{(W,D)}\subset \hlkv_2^{(W,D)}$.
\end{thm}

From this result and the various inclusions we can in particular confirm the aforementioned conjectures in low weight and depth, improving on the previously known bounds.

\begin{cor}
    The Broadhurst-Kreimer conjectures (Conjectures \ref{conj:BK} and  \ref{conj:BK 2}) hold for all pairs $(W,D)$ displayed in Figure \ref{fig:results}, and in particular for all weights $\leq 29$ independent of $D$.
\end{cor}

\begin{cor}
    Conjecture \ref{conj:all iso} holds in all weights up to and including weight 29.
\end{cor}
\begin{proof}
    Here we forget the depth grading. This corresponds to setting $t=1$ in the Broadhurst-Kreimer conjecture. Note that the coefficient $\sum_D \BK_{W,D} = \vdim \free^{(W)}$ of $s^W$ then computes the dimension of the free Lie algebra $\free$ in the given weight $W$.
    Hence if conjecture \ref{conj:BK 2} holds for a given weight $W$ and all $D$ then $\vdim \hlkv_2^{(W)} = \vdim\free^{(W)}$.
    
    Furthermore, since $\gr \hkv_2^{(W)}\subset \hlkv_2^{(W)}$ by Theorem \ref{thm:rs} we have that $\vdim\hkv_2^{(W)} = \vdim \gr \hkv_2^{(W)}\leq \vdim \hlkv_2^{(W)}=\vdim\free^{(W)}$. Hence taking dimensions of the chain \eqref{equ:chain} we get 
    \[
    \vdim\free^{(W)}\leq \vdim \grt_1^{(W)} \leq \vdim \lds^{(W)} \leq \vdim\hkv_2^{(W)} = \vdim\free^{(W)},
    \]
    so that all dimensions must be equal and all inclusions \eqref{equ:chain} must be isomorphisms in that weight. By the same argument it also follows that all inclusions in \eqref{equ:rsdiag} must be isomorphisms in weight $W$ as well.
\end{proof}

\begin{rem}
    There have been several other works in the literature confirming Conjecture \ref{conj:BK} numerically, most notably \cite{datamine}.
    Our contribution is to extend the range of bidegrees in which it is verified, and to cover the Kashiwara-Vergne Lie algebra.
\end{rem}

\begin{figure}
\resizebox{13cm}{!}{
\begin{tabular}{c|c|c|c|c|c|c|c|c|c|c|c|}
    $W$, $D$ & 1 & 2 & 3 & 4 & 5 & 6 & 7 & 8 & 9 & 10 & 11  \\
    \hline
    3 & 1& & & & & & & & & & \\
    4 & 0& & & & & & & & & & \\
    5 & 1& & & & & & & & & & \\
    6 & 0& & & & & & & & & & \\
    7 & 1& & & & & & & & & & \\
    8 & 0& 1& & & & & & & & & \\
    9 & 1& & & & & & & & & & \\
    10 & 0& 1& & & & & & & & & \\
    11 & 1& 0& 1& & & & & & & & \\
    12 & 0& 1& 0& 1& & & & & & & \\
    13 & 1& 0& 2& & & & & & & & \\
    14 & 0& 2& 0& 1& & & & & & & \\
    15 & 1& 0& 2& 0& 1& & & & & & \\
    16 & 0& 2& 0& 3& & & & & & & \\
    17 & 1& 0& 4& 0& 2& & & & & & \\
    18 & 0& 2& 0& 5& 0& 1& & & & & \\
    19 & 1& 0& 5& 0& 5& & & & & & \\
    20 & 0& 3& 0& 7& 0& 3& & & & & \\
    21 & 1& 0& 6& 0& 9& 0& 1& & & & \\
    22 & 0& 3& 0& 11& 0& 7& & & & & \\
    23 & 1& 0& 8& 0& 15& 0& 4& & & & \\
    24 & 0& 3& 0& 16& 0& 14& 0& 1& & & \\
    25 & 1& 0& 10& 0& 23& 0& 11& & & & \\
    26 & 0& 4& 0& 20& 0& 27& 0& 5& & & \\
    27 & 1& 0& 11& 0& 36& 0& 23& 0& 2& & \\
    28 & 0& 4& 0& 27& 0& 45& 0& 16& & & \\
    29 & 1& 0& 14& 0& 50& 0& 48& 0& 7& &  \\
    \hdashline
    30 & 0& 4& 0& 35& 0& 73& 0& 37& 0& 2& ? \\
    31 & 1& 0& 16& 0& 71& 0& 85& 0& 24& ?& \\
    32 & 0& 5& 0& 43& 0& 113& 0& 79& 0& ?& \\
    33 & 1& 0& 18& 0& 96& 0& 147& 0& ?& & \\
    34 & 0& 5& 0& 54& 0& 166& 0& 155& ?& & \\
    35 & 1& 0& 21& 0& 127& 0& 239& 0& ?& & \\
    36 & 0& 5& 0& 66& 0& 239& 0& 281& ?& & \\
    37 & 1& 0& 24& 0& 165& 0& 375& ?& & & \\
    38 & 0& 6& 0& 78& 0& 336& 0& ?& & & \\
    39 & 1& 0& 26& 0& 213& 0& 564& ?& & & \\
    40 & 0& 6& 0& 94& 0& 458& 0& ?& & & \\
    41 & 1& 0& 30& 0& 266& 0& 834& ?& & & \\
    42 & 0& 6& 0& 111& 0& 615& 0& ?& & & \\
    43 & 1& 0& 33& 0& 333& 0& 1190& ?& & & \\
    44 & 0& 7& 0& 128& 0& 814& ?& & & & \\
    45 & 1& 0& 36& 0& 409& 0& ?& & & & \\
    46 & 0& 7& 0& 150& 0& 1055& ?& & & & \\
    47 & 1& 0& 40& 0& 498& 0& ?& & & & \\
    48 & 0& 7& 0& 173& 0& 1354& ?& & & & \\
    49 & 1& 0& 44& 0& 600& 0& ?& & & & \\
    50 & 0& 8& 0& 196& 0& 1717& ?& & & & \\
    51 & 1& 0& 47& 0& 720& 0& ?& & & & \\  
    52 & 0& 8& 0& 224& 0& 2149& ?& & & & \\
    53 & 1& 0& 52& 0& 851& 0& ?& & & & \\ 
    54 & 0& 8& 0& 254& 0& 2666& ?& & & & \\
    55 & 1& 0& 56& 0& 1005& 0& ?& & & & \\ 
    56 & 0& 9& 0& 284& 0& 3281& ?& & & & \\
    57 & 1& 0& 60& 0& 1176& 0& ?& & & & \\ 
    58 & 0& 9& 0& 320& 0& 3994& ?& & & & \\
    59 & 1& 0& 65& 0& 1368& 0& ?& & & & \\ 
    60 & 0& 9& 0& 357& 0& 4834& ?& & & & \\
    61 & 1& 0& 70& 0& 1582& ?& & & & & \\ 
    62 & 0& 10& 0& 395& 0& ?& & & & & \\
    63 & 1& 0& 74& 0& 1824& ?& & & & & \\ 
    64 & 0& 10& 0& 439& 0& ?& & & & & \\
    70 & 0& 11& 0& 585& 0& ?& & & & & \\
\end{tabular}
\hspace{1cm}
\begin{tabular}{c|c|c|c|c|c|c|}
    $W$, $D$ & 1 & 2 & 3 & 4 & 5 & 6\\
    \hline
    66 & 0& 10& 0& 485& 0& ? \\
    67 & 1& 0& 85& 0 & 2381& ? \\
    68 & 0& 11& 0& 531& 0& ? \\
    69 & 1& 0& 90& 0& 2703& ? \\ 
    70 & 0& 11& 0& 585& 0& ? \\
    71 & 1& 0& 96& 0& 3057& ? \\ 
    72 & 0& 11& 0& 640& 0& ? \\ 
    73 & 1& 0& 102& 0& 3444& ? \\ 
    74 & 0& 12& 0& 696& 0& ? \\ 
    75 & 1& 0& 107& 0& 3871& ? \\ 
    76 & 0& 12& 0& 759& 0& ? \\ 
    77 & 1& 0& 114& 0& 4328& ? \\ 
    78 & 0& 12& 0& 825& 0& ? \\ 
    79 & 1& 0& 120& 0& 4833& ? \\ 
    80 & 0& 13& 0& 891& 0& ? \\
    81 & 1& 0& 126& 0& 5376& ?  \\ 
    82 & 0& 13& 0& 966& 0&  ?\\ 
    83 & 1& 0& 133& 0& 5964& ? \\ 
    84 & 0& 13& 0& 1042& 0& ? \\ 
    85 & 1& 0& 140& 0& 6598& ? \\ 
    86 & 0& 14& 0& 1120& ?&  \\
    87 & 1& 0& 146& 0& ?&  \\ 
    88 & 0& 14& 0& 1206& ?&  \\ 
    89 & 1& 0& 154& 0& ?&  \\ 
    90 & 0& 14& 0& 1295& ?&  \\ 
    91 & 1& 0& 161& 0& ?&  \\ 
    92 & 0& 15& 0& 1384& ?&  \\ 
    93 & 1& 0& 168& 0& ?&  \\ 
    94 & 0& 15& 0& 1484& ?&  \\ 
    95 & 1& 0& 176& 0& ?&  \\ 
    96 & 0& 15& 0& 1585& ?&  \\ 
    97 & 1& 0& 184& 0& ?&  \\ 
    98 & 0& 16& 0& 1688& ?&  \\ 
    99 & 1& 0& 191& 0& ?&  \\ 
    100 & 0& 16& 0& 1800& ?&  \\ 
    101 & 1& 0& 200& 0& ?&  \\ 
    102 & 0& 16& 0& 1916& ?&  \\ 
    103 & 1& 0& 208& 0& ?&  \\ 
    104 & 0& 17& 0& 2032& ?&  \\ 
    105 & 1& 0& 216& 0& ?&  \\ 
    106 & 0& 17& 0& 2160& ?&  \\ 
    107 & 1& 0& 225& 0& ?&  \\ 
    108 & 0& 17& 0& 2289& ?&  \\ 
    109 & 1& 0& 234& 0& ?&  \\ 
    110 & 0& 18& 0& 2421& ?&  \\ 
    111 & 1& 0& 242& 0& ?&  \\ 
    112 & 0& 18& 0& 2563& ?&  \\ 
    113 & 1& 0& 252& 0& ?&  \\ 
    114 & 0& 18& 0& 2709& ?&  \\ 
    115 & 1& 0& 261& 0& ?&  \\ 
    116 & 0& 19& 0& 2855& ?&  \\ 
    117 & 1& 0& 270& 0& ?&  \\ 
    118 & 0& 19& 0& 3015& ?&  \\ 
    119 & 1& 0& 280& 0& ?&  \\ 
    120 & 0& 19& 0& 3176& ?&  \\ 
    121 & 1& 0& 290& 0& ?&  \\ 
    122 & 0& 20& 0& 3340& ?&  \\ 
    123 & 1& 0& 299& 0& ?&  \\ 
    124 & 0& 20& 0& 3515& ?&  \\ 
    125 & 1& 0& 310& 0& ?&  \\ 
    126 & 0& 20& 0& 3695& ?&  \\ 
    127 & 1& 0& 320& 0& ?&  \\
    \hline
    \multicolumn{1}{c}{} \\
\end{tabular}
}
    \caption{\label{fig:results} The $\Q$-dimensions $\vdim \lkv_2^{(W,D)}$ of the bigraded pieces of $\widehat \lkv_2$ in weight $W$ and depth $D$.
    The empty entries above the dashed line ($W\leq 29$) have been computed and are all zero, for all depths $1,\dots ,W$, although the zeroes are not displayed explicitly.
    Below the dashed line ($W\geq 30$) only the low depth entries up to the "?" have been computed, and are unknown to the right of and including the "?". All computed values agree with those predicted by Conjecture \ref{conj:BK}.}
\end{figure}

\section{The linearized Kashiwara-Vergne Lie algebra (recollections)}
In this section we recall the definition of the Kashiwara-Vergne Lie algebra and its linearized version, see \cite{AlekseevTorossian, RaphaelSchneps} for more details.
\subsection{Free Lie algebra and cyclic words}
Let $\f_2=\FreeLie(X,Y)$ be the free Lie algebra in symbols $X$ and $Y$.
It is naturally bigraded with the weight grading counting the total number of symbols $X$, $Y$, and the depth grading counting the number of occurrences of $Y$ only. Denoting by $\f_2^{(W)}\subset \f_2$ (resp. $\f_2^{(W,D)}\subset \f_2$) the component of fixed weight $W$ (resp. fixed weight and depth $(W,D)$) we have the well-known dimension formulas, with $\mu(-)$ the Möbius function,
\begin{align}\label{equ:dim flie}
    \vdim \f_2^{(W)} &= \frac 1W \sum_{d\mid W} \mu(d) \, 2^{W/d} 
    &
    \vdim \f_2^{(W,D)} &= \frac 1W \sum_{d\mid W \atop d\mid D} 
    \mu(d)\,
    \binom{W/d}{D/d} .
\end{align}
Let $\as_2$ be the free associative algebra in symbols $X$ and $Y$. Then there is a natural map of bigraded Lie algebras $\f_2\hookrightarrow \as_2$.
Furthermore, for $A$ some word in $X$ and $Y$ we define the operator $\partial_A : \as_2\to\as_2$ such that on words $W\in \as_2$
\[
\partial_A W = 
\begin{cases}
    B & \text{if $W=AB$ for some word $B$} \\
    0 & \text{otherwise}
\end{cases}.
\]

Next, we denote by $\cy_2=\as_2/[\as_2,\as_2]$ the vector space spanned by all cyclic words in letters $X$ and $Y$. We have maps 
\[
\as_2 \xrightarrow{\pi} \cy_2 \xrightarrow{\Sigma} \as_2,
\]
with the left-hand map the obvious projection and $\Sigma$ defined on a cyclic word $(W_0\cdots W_n)$ (with $W_j\in \{X,Y\}$) as
\[
\Sigma W = \sum_{j=0}^n W_jW_{j+1}\cdots W_nW_0\cdots W_{j-1}.
\]
We may then define on $\cy_2$ the operators 
\begin{align*}
\Delta_X &:= \pi \circ \partial_{XX} \circ \Sigma : \cy_2\to \cy_2 \\
\Delta_Y &:= \pi \circ \partial_{YY} \circ \Sigma : \cy_2\to \cy_2 \\
\Delta &:= \Delta_X+\Delta_Y : \cy_2\to \cy_2.
\end{align*}
Furthermore, we may define on $\cy_2$ three Lie brackets. The first is defined as follows:
\begin{gather*}
[-,-]_X : \cy_2\times \cy_2 \to \cy_2 \\
[A,B]_X:= ([X,\partial_X(\Sigma A)](\partial_X(\Sigma B))).  
\end{gather*}
The second Lie bracket $[-,-]_Y$ is defined in the same way by interchanging $X$ and $Y$,
\begin{gather*}
[A,B]_Y:= ([Y,\partial_Y(\Sigma A)](\partial_Y(\Sigma B))).  
\end{gather*}
The third is
\[
[A,B] = [A,B]_X + [A,B]_Y.
\]
In fact, any linear combination would be a Lie bracket.
To ensure compatibility with these Lie brackets we define the weight grading on $\cy_2$ such that a cyclic word of length $W+1$ has weight $W$.
We also define the depth grading as the number of $Y$ symbols occurring in words minus one. For example
\begin{align*}
(YXXYX)=(XYXXY)&\in \cy_2^{(4,1)}.
\end{align*}
All three Lie brackets respect the weight grading, and $[-,-]_Y$ also respects the depth grading. However, $[-,-]$ only preserves the associated depth filtration.


\subsection{Tangential and special derivations}
A tangential derivation of the Lie algebra $\f_2$ is a derivation $G:\f_2\to \f_2$ defined on generators as 
\begin{align*}
    G(X) &= [X,G_1] & G(Y)&=[Y,G_2] 
\end{align*}
for some $G_1\in \f_2$ and $G_2\in \f_2$. The tangential derivations naturally form a Lie algebra with the commutator as the Lie bracket.
We denote this Lie algebra of tangential derivations by $\tder_2$. As a vector space we have $\tder_2\cong \f_2\times \f_2$.
There is a natural bigrading on $\tder_2$ such that $(G_1,G_2)\in \tder_2$ is homogeneous of weight $W$ and depth $D$ if $G_1,G_2\in \f_2^{(W,D)}$.

A tangential derivation $(G_1,G_2)\in \tder_2$ is called {\em special} if 
\begin{equation}\label{equ:sder}
[X,G_1] + [Y,G_2]=0.
\end{equation}
Note that the map $[X,-]:\f_2\to \f_2$ is injective in weights $\geq 2$ and hence $G_1$ is uniquely determined once $G_2$ is known.

We denote the subspace of special derivations by $\sder_2\subset\tder_2$.
They form a Lie subalgebra. The weight grading is inherited by $\sder_2$ since the defining identity \eqref{equ:sder} is homogeneous with respect to weight. The latter equation is a priori not homogeneous with respect to the depth grading. However, we define $(G_1,G_2)$ to be homogeneous of depth $D$ if $G_2$ is a linear combination of terms with $D$ symbols $Y$ and $G_1$ is a linear combination of terms with $D+1$ symbols $Y$.

Since the map $\tder_2\to \f_2$, $(G_1,G_2)\mapsto [X,G_1] + [Y,G_2]$ is surjective in weights $\geq 2$ one may easily obtain the dimension formula 
\[
\vdim \sder_2^{(W,D)} = 
\vdim \f_2^{(W,D+1)} + \vdim \f_2^{(W,D)} - \vdim \f_2^{(W+1,D+1)},
\]
where the right-hand side can be computed explicitly via \eqref{equ:dim flie}.

\subsection{The Kashiwara-Vergne Lie algebra}\label{sec:kv def}
We define the map $\iota: \sder_2 \to \cy_2$ on a homogeneous element $(G_1,G_2)\in \sder_2^{(W,D)}$ as
\begin{gather*}
    \frac{1}{D+1} (Y G_2) \in \cy_2^{(W,D)}.
\end{gather*}
\begin{prop}
    The map $\iota$ is an injective map of (weight-)graded Lie algebras.
\end{prop}
\begin{proof}
We first show injectivity. Set $\alpha = \frac{1}{1+W}\pi ( X G_1 + Y G_2 )$. Using $XG_1 + YG_2 = G_1X + G_2Y$ we obtain $\Sigma(\alpha) = X G_1 + Y G_2$. Moreover, using that $\pi(y \partial_Y \Sigma(\beta)) = (D+1) \beta$ for any $\beta \in \cy_2^{(W,D)}$ we first obtain
\[
\alpha = \frac{1}{D+1} \pi(y \partial_Y \Sigma(\alpha)) = \frac{1}{D+1} \pi(y \partial_Y (X G_1 + Y G_2 )) = \frac{1}{D+1} \pi(y G_2),
\]
and hence $G_1 = \partial_X \Sigma (\frac{1}{D+1} \pi(y G_2))$ and $G_2 = \partial_Y \Sigma (\frac{1}{D+1} \pi(y G_2))$. Finally, it follows from \cite[Proposition 1.4]{vandenBergh} that the map $H \colon \cy_2 \to \mathbf{Der}(\as_2)$ given by $H(\alpha)(X) = [X, \partial_X \Sigma(\alpha)], H(\alpha)(Y) = [Y, \partial_Y \Sigma(\alpha)]$ is a map of Lie algebras.
\end{proof}

\begin{defi}
    The Kashiwara-Vergne Lie algebra $\kv_2\subset \sder_2$ is the kernel 
    \[
    \kv_2 := \ker (\Delta\circ \iota),
    \]
    equipped with the Lie bracket $[-,-]$ of $\sder_2$.
    The linearized Kashiwara-Vergne Lie algebra $\kv_2\subset \sder_2$ is the kernel 
    \[
    \lkv_2 := \ker (\Delta_Y\circ \iota),
    \] 
    equipped with the Lie bracket $[-,-]_Y$.
    The extended linearized Kashiwara-Vergne Lie algebra $\hlkv_2\subset \sder_2$ is the kernel 
    \[
    \hlkv_2 := \ker (\pi_{\geq 2}\Delta_Y\circ \iota)
    \]
    where $\pi_{\geq 2}:\cy_2\to \cy_2$ is the projection onto the part of depth $\geq 2$.
    The extended Kashiwara-Vergne Lie algebra $\hkv_2\subset \sder_2$ is the preimage 
    \[
    \hkv_2 := (\Delta\circ \iota)^{-1}(V),
    \]
    of the subspace $V\subset \cy_2$ spanned by the linear combinations of cyclic words $(X+Y)^n -X^n-Y^n$.
\end{defi}
We have natural inclusions of Lie algebras with the standard bracket $[-,-]$
\[
\kv_2 \subset \hkv_2\subset \sder_2\subset \cy_2
\]
and natural inclusions of Lie algebras with respect to the bracket $[-,-]_Y$
\[
\gr \kv_2 \subset \lkv_2\subset \hlkv_2\subset \sder_2\subset \cy_2.
\]
\begin{ex}\label{ex:sigma}
The depth 1 elements $\bar \sigma_{2k+1}=(\frac12 \sum_{j=0}^{2k-1} [\ad_X^j Y, \ad_X^{2k-1-j}Y ], \ad_X^{2k} Y)\in \sder_2$ satisfy 
\[
\iota \bar \sigma_{2k+1} = \frac 12(Y \ad_X^{2k} Y) \in \cy_2
\]
and by degree (depth) reasons we have
\[
\Delta_Y \iota \bar \sigma_{2k+1}= (X^{2k}Y) \in \cy_2^{(W,1)}.
\]
Hence $\bar \sigma_{2k+1}\in \hlkv_2$.
\end{ex}

Generally, we have $\lkv_2^{(W,1)}=0$ for all $W$ and 
\[
\hlkv_2^{(W,D)} = 
\begin{cases}
    \Q \bar\sigma_{2k+1} & \text{if $W=2k+1$ and $D=1$} \\
    0 & \text{if $W=2k$ and $D=1$} \\
    \lkv_2^{(W,D)} & \text{if $D\geq 2$}
\end{cases}.
\]
Hence for our numerical experiments it is sufficient to compute the dimension of $\lkv_2^{(W,D)}$, the dimension of $\hlkv_2^{(W,D)}$ is then determined by the above formula.

\begin{rem}
    The above definition of $\lkv_2$ coincides with the one given in \cite[Definition 5]{RaphaelSchneps}. In loc. cit. two conditions on $G_2 \in \f_2$ (or $b$ in their notation) are given. The first condition (\emph{push-invariance}) is ensuring that there exists a corresponding $G_1$ and is equivalent to $G_2 = \partial_Y \tfrac{1}{D+1}\Sigma \pi(YG_2)$. The second condition (\emph{circ-neutrality}) is $\partial_Y \Sigma \pi( \partial_Y G_2) = 0$ whenever $G_2$ is in depth $\geq 2$. But this is equivalent to $\pi( \partial_Y G_2) = \tfrac{1}{D+1}\pi ( \partial_{YY} \Sigma \pi (YG_2) ) = \tfrac{1}{D+1} \Delta_Y (YG_2)$.
\end{rem}

\subsection{Dimension of $\lkv_2^{(W,D)}$}
From the definition we see that we may compute the dimensions of the bigraded pieces $\lkv_2^{(W,D)}$ of $\lkv_2$ as 
\[
\vdim\lkv_2^{(W,D)} = 
\vdim \sder_2^{(W,D)}
- 
rank\left(\sder_2^{(W,D)} \xrightarrow{\Delta_Y\iota}  \cy_2^{(W-1,D-1)}\right).
\]
However, this requires that we can write down a basis or generating set of $\sder_2^{(W,D)}$. We do not know an efficient construction of a basis.
But we may use the following lemma to get a generating set efficiently.
\begin{lemma}\label{lem:genset}
    Let $W_1,W_2\geq 1$ be numbers such that $W_1+W_2=W+1$. Then the image of $\iota:\sder_2^{(W,D)}\to \cy_2^{(W,D)}$ is the same as the image of the map 
    \begin{equation}\label{equ:genset}
    \begin{gathered}
    \iota_{W_1,W_2, D} \colon \bigoplus_{D_1+D_2=D+1} \f_2^{(W_1,D_1)}\otimes \f_2^{(W_2,D_2)} \to \cy_2^{(W,D)} \\
    a\otimes b \mapsto (ab).
    \end{gathered}
    \end{equation}
\end{lemma}
In words, the linear combination of cyclic words $(ab)$ is formed by first applying the inclusion $\f_2\to\as_2$ to $a$ and $b$, then concatenating the resulting words to $ab$, then projecting to cyclic words.
\begin{proof}
    The proof of \cite[Proposition 6.1]{Drinfeld} shows that the map $\iota$ factors as $\sder_2 \to \mathrm{Sym}^2(\f_2)_{\f_2} \to \cy_2$, where the first map is an isomorphism. But since $\f_2^{(W+1, D)} = [X, \f_2^{(W,D)}] + [Y, \f_2^{(W, D-1)}]$ we obtain that
    \[
    \bigoplus_{D_1 + D_2 = D+1} \f_2^{(W_1+1,D_1)} \otimes \f_2^{(W_2-1,D_2)} 
    \]
    and
    \[
    \bigoplus_{D_1 + D_2 = D+1} \f_2^{(W_1,D_1)} \otimes \f_2^{(W_2+1,D_2)} 
    \]
    have the same image in coinvariants of $\f_2 \otimes \f_2$.
\end{proof}
Hence we have that for any choice of $W_1,W_2\geq 1$ such that $W_1+W_2=W+1$
\begin{equation}\label{equ:lkvdim}
\vdim\lkv_2^{(W,D)} = 
\vdim \sder_2^{(W,D)}
- 
rank( \Delta_Y\circ\iota_{W_1,W_2, D}).
\end{equation}

\section{Computational methods}
We proceed in two separate steps to compute the numbers $C_{W,D}:=\mathrm{dim}( \hlkv_2^{(W,D)})$: We separately compute a lower bound $\underline C_{W,D}$ on $C_{W,D}$ and an upper bound $\overline C_{W,D}$. Both bounds happen to agree, $\underline C_{W,D}=\overline C_{W,D}$ for all pairs $(W,D)$ displayed in Figure \ref{fig:results}, and hence Theorem \ref{thm:main} follows.

\subsection{Computation of the upper bound}
We get an upper bound by using \eqref{equ:lkvdim}, computing the rank over prime fields to avoid memory issues.
More precisely, we proceed as follows:
\begin{enumerate}
    \item We use Lemma \ref{lem:genset} to produce a generating set of $\sder_2^{W,D}$ by computing a list of pairs of Lyndon words $(A,B)$ such that $W=|A|+|B|-1$ and of total depth $D+1$. In practice, it was most economic to choose the lengths $|A|$, $|B|$ to be approximately $W/2$ each.
    \item We also produce a list of all cyclic words of length $W$ and depth $D$.
    \item For each pair of Lyndon words we compute the image in cyclic words of the corresponding $\sder_2$-element (see \eqref{equ:genset}), and apply the half-divergence operator $\Delta_Y$.
    The result is a linear combination of cyclic words of length $W$ and depth $D-1$, that yields one row of a matrix $A'$, which we reduce to mod $p$ coefficients, for $p$ a suitable prime number. Concretely, we chose $p=3323$.
    \item The matrix $A'$ is compressed to a square matrix $A$ of size $(\vdim\sder_2^{W,D})\times (\vdim\sder_2^{W,D})$, likely retaining the rank. This is done by randomly adding multiples of the last $k$ columns (resp. rows) to the first columns (rows). This yields a dense matrix $A$ of shape $(\vdim\sder_2^{W,D})\times (\vdim\sder_2^{W,D})$ with entries in $\mathbb F_p$.
    \item The rank $r_A$ of the matrix $A$ (mod $p$) is computed using straightforward Gaussian elimination. Our upper bound is then 
    \[
    \overline C_{W,D} := \vdim\sder_2^{W,D} - r_A.
    \]
\end{enumerate}
We note that the matrix $A'$ is in fact never computed in full as it would typically not fit into memory. Instead the compression (step 4) is performed directly after computing each matrix row, greatly reducing memory pressure and runtime.
Also note that both the use of modular arithmetic and the compression step 4 may reduce the rank of the matrix, but never increase it. Hence we indeed have that 
\[
    \overline C_{W,D} \geq  C_{W,D}
\]
as required. Also note that all steps of the algorithm allow for obvious parallelization. 

\subsection{Computation of the lower bound}
Note that we know the following classes of $\hlkv_2$: 
\begin{itemize}
    \item The elements $\bar \sigma_{2k+1}$ of Example \ref{ex:sigma} in depth 1 and weight $2k+1$. 
    \item For each period polynomial $P$ the Brown elements
    \[
    \bar \rho_P
    \]
    of weight $|P|+2$ and depth 4, see \cite{BrownDepth} for the definition. They are in fact in the subspace $\lds\subset\hlkv_2$ as shown by Brown \cite{BrownDepth}.
\end{itemize}
Now we can produce a supply of elements of $\hlkv_2^{(W,D)}$ by computing all possible Lie expressions of the elements $\bar \sigma_{2k+1}$ and $\bar \rho_P$ of suitable weight and degree. Let $\psi_1,\dots,\psi_N\in \hlkv_2^{(W,D)}$ be the list of all elements thus produced. Then a lower bound on the dimension of $\hlkv_2^{(W,D)}$ is obtained by the dimension of the span of these elements. Algorithmically, we proceed as follows:
\begin{enumerate}
    \item Produce a list of all cyclic words of length $W$ and depth $D$.
    \item By a straight-forward recursion produce the list of generators $\psi_1,\dots,\psi_N\in \hlkv_2^{(W,D)}$ as above.
    \item Compute the matrix $B'$ whose $j$-th row is the expression of $\psi_j$ in the basis of cyclic words produced in step 1.
    \item Compress $B'$ to a smaller matrix $B$ as in the previous subsection and compute its rank $r_{B}$ modulo a prime number $p$. Our lower bound is then $\underline C_{W,D}:=r_B$.
\end{enumerate}
We have that the dimension of the span of $\psi_1,\dots,\psi_N$ is at least $r_B$, so that indeed 
\[
\underline C_{W,D} \leq C_{W,D}.
\]

\subsection{Hardware and runtime}
We ran our program on the ETH math department's 64-core AMD EPYC servers with 512 GB of memory.
The runtimes depend on the weight-depth pair $(W,D)$ and the degree of parallelism employed. We did not record runtimes for each individual $(W,D)$. Since the machines are shared with other users they also depend on the current load.
However, the dominant contribution to the runtime comes from the computation of the rank of the dense matrix $A$ above, at least for higher depths.
For example, for $(W,D)=(29,11)$ we have that $\vdim \sder_2^{(W,D)}=99591$, so that the matrix $A$ has size $99591\times 99591$. We produce $191931$ generating vectors and compute the matrix $A$ in roughly 18 minutes on 64 cores. The rank computation took 10h45 on 64 cores. For the hardest entry in weight 29 (depth 14) the computation of the $178305\times 178305$-matrix $A$ took 1h16, and the rank computation around 4 days, each using (up to) 64 cores.
For the next higher weight 30, that we did not compute fully, the hardest entry (depth 15) would have produced a square matrix of size $322938$, with an estimated runtime of around a month, which seemed excessive. The partial computations in low depth and high weight were less computationally expensive, with the hardest pairs $(W,D)$ computed ($D=5$) taking a few hours.  
The matrix $A$ is also mainly responsible for the memory footprint, as otherwise no large objects have to be held in memory.
The interested reader can find the source code of our programs here:

\begin{center}
\url{https://github.com/wilthoma/lkv}.
\end{center}

\end{document}